\documentclass[11pt]{article}

\usepackage{amsmath}
\usepackage{amsthm}
\usepackage{amsfonts}
\usepackage{setspace}
\usepackage{fullpage}
\usepackage{amssymb}
\usepackage{enumitem}
\usepackage{bbold} 

\usepackage{comment}

\usepackage{tikz}
\usetikzlibrary{patterns,snakes} 
\usepackage{float}
\usepackage{caption}

\newtheorem{lemma}{Lemma}
\newtheorem{theorem}{Theorem}

\newtheorem{remark}{Remark}

\newcommand{\K}{{\cal K}}
\newcommand{\tih}{\tilde{h}}
\newcommand{\talp}{\tilde{\alpha}}

\newcommand{\floor}[1]{\left\lfloor{#1}\right\rfloor}
\newcommand{\ceil}[1]{\left\lceil{#1}\right\rceil}
\date{}
\newcounter{num}

  \DeclareMathOperator{\Forb}{Forb}

\title{Large homogeneous subgraphs in bipartite graphs with forbidden induced subgraphs}

\author{Maria Axenovich\thanks{Karlsruhe Institute of Technology, Karlsruhe, Germany, \texttt{maria.aksenovich@kit.edu}.} \and Casey Tompkins\thanks{Karlsruhe Institute of Technology, Karlsruhe, Germany, \texttt{ctompkins496@gmail.com}.} \and Lea Weber\thanks{Karlsruhe Institute of Technology, Karlsruhe, Germany, \texttt{0815lea@gmx.net}.}}

\begin{document}

\maketitle

\abstract{For a bipartite graph $G$, let $\tih(G)$ be the largest $t$ such that  either $G$ contains $K_{t,t}$, a  complete bipartite subgraph with sides of size $t$ each or a bipartite complement of $G$ contains $K_{t,t}$. 
For a class $\cal{F}$, let  $\tih({\cal F})= \min\{\tih(G): G\in {\cal F}\}$. We say that a bipartite graph $H$ is strongly acyclic if neither $H$ nor its bipartite complement contain a cycle.  By $\Forb(n, H)$ we denote a set of 
bipartite graphs with parts of sizes $n$ each, that do not contain $H$ as an induced bipartite subgraph respecting the sides.
One can easily show that $\tih( \Forb(n,H))= O(n^{1-\epsilon})$ for a positive $\epsilon$ if $H$ is not strongly acyclic. 
Here, we prove that $\tih(\Forb(n, H))$ is linear in $n$ for all strongly acyclic graphs except for four graphs. }


\section*{Introduction}

A conjecture of Erd\H{o}s and Hajnal \cite{EH} states that for any graph $H$ there is a constant $\epsilon>0$ such that any $n$-vertex graph that does not contain $H$ as an induced subgraph has either a clique or a coclique on at least $n^{\epsilon}$ vertices. 
While this conjecture remains open, see for example  a survey by Chudnovsky \cite{C}, we address the bipartite setting of the problem.\\

Let $G$ be a bipartite graph with parts $U,V$ of size $n$ each, we write $G=((U,V), E)$, and further write $E=E(G)$.
We shall often depict the sets $U$ and $V$  as  sets of points on two horizontal lines in the plane and call $U$ the top part and $V$ the bottom part. We say that a graph is the {\it bipartite complement } of $G$ if it has the same vertex set as $G$ and its edge set is $(U\times V)\setminus E$. We denote the bipartite complement of a graph $G$ by $G'$.
By $\tilde\omega(G)$ we denote the largest integer $t$ such that there are $A \subseteq U, B \subseteq V$ with $|A|= |B| = t$ and $ab \in E$ for all $a\in A, b\in B$, i.e.,  $A$ and $B$ form a {\it biclique}. 
By $\tilde\alpha(G)$ we denote the largest integer $t$ such that there are $A \subseteq U, B \subseteq V$ with $|A|= |B| = t$ and $ab \not\in E $ for all $a\in A, b\in B$, i.e., $A$ and $B$ form a {\it co-biclique}. 	Let $\tilde{h}(G) = \max \{\tilde{\alpha}(G), \tilde{\omega}(G)\}$.\\

For a bipartite graph $H=((U,V), E)$ and a bipartite graph $G=((A,B), E')$ we say that $H$ is an {\it induced 
bipartite subgraph } of $G$ {\it respecting sides} if $U\subseteq A$, $V\subseteq B$, and for any $u\in U, v\in V$, $uv\in E(H)$ if and only if $uv\in E(G)$. An induced subgraph $H^*$ of $G$ is a {\it copy } of $H$ in $G$ if $H^*$ is isomorphic to $H$ such that the isomorphic image of $U$ is contained in $A$ and the isomorphic image of $V$ is contained in $B$.
Let $\Forb(n, H)$ denote the set of  all bipartite graphs  with parts of size $n$ each that do not contain a copy of $H$ as an induced  bipartite subgraph respecting sides. We call a bipartite graph $H$-free if it does  not contain a copy of $H$.
Let 
\[\tih(n,H)= \tih(\Forb(n, H)) = \min\{\tih(G): ~G\in \Forb(n,H)\}.\]

It is implicit from a result shown by  Erd\H{o}s, Hajnal, and Pach \cite{EHP}, that for any bipartite $H$ with the smallest part of size $k$, that 
$\tih(n, H)=\Omega(n^{1/k})$.  A standard probabilistic argument shows that  if $H$ or its bipartite complement contains 
a cycle, then $\tih(n, H) = O(n^{1-\epsilon})$ for a positive $\epsilon$.  Here, we address the question of when $\tih(n,H)$ is linear in $n$.  We say that a bipartite graph $H$ is {\it strongly acyclic}  if neither $H$ nor its bipartite complement contain a cycle.  It is not difficult to show that  $\tih(n,H)$ could be linear only if $H$ is strongly acyclic. We show that for all but at most four strongly acyclic graphs $H$, $\tih(n,H)$ is linear in $n$.
Moreover, for several graphs $H$ we determine $\tih(n,H)$ exactly.

\begin{theorem}\label{main} There is a set ${\cal H}$  of at most four graphs such that for any strongly acyclic bipartite graph $H$,  such that neither $H$ nor $H'$ is in  ${\cal H}$, there is a positive constant $c=c(H)$ such that $\tih(n, H) \geq cn$.
\end{theorem}

The set ${\cal H}$ is given in Figure \ref{acyclic-rest}.\\

Note that the notion of large bicliques and co-bicliques in ordered bipartite graphs  with forbidden induced subgraphs corresponds to the notion of submatrices of all $0$'s or of all $1$'s in binary matrices with forbidden submatrices. A paper of  Kor\'andi, Pach, and Tomon   \cite{KPT}  addresses a similar question for  matrices. In addition, one could interpret bipartite graphs as set systems consisting of all the neighborhoods of vertices from one part. Structural properties of these graphs  in terms of VC-dimension of the respective system and connection to Erd\H{o}s-Hajnal conjecture are addressed for example by Fox, Pach, and Suk \cite{FPS}.\\

The paper is structured as follows. In Section \ref{general} we show the general bounds for completeness.
In Section \ref{acyclic} we characterize all strongly acyclic bipartite graphs.
In Section \ref{proof} we find linear lower bounds on $\tih(n,H)$ for each of the strongly acyclic graphs with few exceptions, thus proving Theorem \ref{main}. Finally, in Section \ref{small} we determine the constant $c$ in Theorem \ref{main} exactly for forbidden bipartite graphs with two vertices in each part.

\section{Characterization of all strongly acyclic graphs}\label{acyclic}

In this section we list all strongly acyclic bipartite graphs up to a bipartite complement.

\begin{figure}[H] 
	\centering
	\begin{minipage}{0.20\linewidth}
		\centering
		\begin{tikzpicture}[scale=0.45]
		\node[draw, circle, inner sep=2pt] (1) at (0,2.5){};
		\node[draw, circle, inner sep=2pt] (2) at (2,2.5){};
		\node[draw, circle, inner sep=2pt] (3) at (4,2.5){};
		
		\node[draw, circle, inner sep=2pt] (4) at (0,0){};
		\node[draw, circle, inner sep=2pt] (5) at (2,0){};
		\node[draw, circle, inner sep=2pt] (6) at (4,0){};
		
		\draw[thick] (4)--(1)--(5)--(2)--(6);
		\end{tikzpicture}
		\caption*{$P_5'$}
	\end{minipage}
	\begin{minipage}{0.20\linewidth}
		\centering
		\begin{tikzpicture}[scale=0.45]
		\node[draw, circle, inner sep=2pt] (1) at (0,2.5){};
		\node[draw, circle, inner sep=2pt] (2) at (2,2.5){};
		\node[draw, circle, inner sep=2pt] (3) at (4,2.5){};
		
		\node[draw, circle, inner sep=2pt] (4) at (0,0){};
		\node[draw, circle, inner sep=2pt] (5) at (2,0){};
		\node[draw, circle, inner sep=2pt] (6) at (4,0){};
		
		\draw[thick] (1)--(4)--(2)--(5)--(3)--(6);
		\end{tikzpicture}
		\caption*{$P_6$}
	\end{minipage}
	\begin{minipage}{0.22\linewidth}
		\centering
		\begin{tikzpicture}[scale=0.45]
		\node[draw, circle, inner sep=2pt] (1) at (1,2.5){};
		\node[draw, circle, inner sep=2pt] (2) at (3.5,2.5){};
		\node[draw, circle, inner sep=2pt] (3) at (6,2.5){};
		
		\node[draw, circle, inner sep=2pt] (4) at (1,0){};
		\node[draw, circle, inner sep=2pt] (5) at (2.25,0){};
		\node[draw, circle, inner sep=2pt] (6) at (3.5,0){};
		\node[draw, circle, inner sep=2pt] (7) at (6,0){};
		
		\draw[thick] (1)--(4)--(2)--(6)--(3)--(7) (2)--(5) ;
		\end{tikzpicture}
		\caption*{$H_{3,4}^1$}
	\end{minipage}
	\begin{minipage}{0.22\linewidth}
		\centering
		\begin{tikzpicture}[scale=0.45]
		\node[draw, circle, inner sep=2pt] (1) at (1,2.5){};
		\node[draw, circle, inner sep=2pt] (2) at (3,2.5){};
		\node[draw, circle, inner sep=2pt] (3) at (5,2.5){};
		
		\node[draw, circle, inner sep=2pt] (4) at (0,0){};
		\node[draw, circle, inner sep=2pt] (5) at (2,0){};
		\node[draw, circle, inner sep=2pt] (6) at (4,0){};
		\node[draw, circle, inner sep=2pt] (7) at (6,0){};
		
		\draw[thick] (4)--(1)--(5)--(2)--(6)--(3)--(7);
		\end{tikzpicture}
		\caption*{$P_7$}
	\end{minipage}
	\caption{Strongly acyclic subgraphs with parts of sizes at least $3$ }\label{acyclic-rest}
	\end{figure}

\begin{figure}[H]

	\centering
	\begin{tikzpicture}[scale=0.45]
		\node[draw, circle, inner sep=1.5pt] (1) at (2,4){};
		\node[draw, circle, inner sep=1.5pt] (2) at (7,4){};
		\node[draw, circle, inner sep=1.5pt] (3) at (-1,0){};
		\node[draw, circle, inner sep=1.5pt] (4) at (0,0){};
		\node[draw, circle, inner sep = 1.5pt] (5) at (3,0){};
		\node[draw, circle, inner sep = 0pt] at (1,0){};
		\node[draw, circle, inner sep = 0pt] at (1.5,0){};
		\node[draw, circle, inner sep = 0pt] at (2,0){};
		
		\node[draw, circle, inner sep=1.5pt] (13) at (5,0){};
		\node[draw, circle, inner sep=1.5pt] (14) at (6,0){};
		\node[draw, circle, inner sep = 1.5pt] (15) at (9,0){};
		\node[draw, circle, inner sep = 0pt] at (7,0){};
		\node[draw, circle, inner sep = 0pt] at (7.5,0){};
		\node[draw, circle, inner sep = 0pt] at (8,0){};
		
		\node[] (7) at (-1.5,0){};
		\node[] (8) at (3.5,0){};
		\draw [thick, decoration={brace, mirror, raise=0.3cm}, decorate] (7)--(8);
		\node [] at (1,-1.5) {$s_1$}; 
		
		\node[] (9) at (4.5,0){};
		\node[] (10) at (9.5,0){};
		\draw [thick, decoration={brace, mirror, raise=0.3cm}, decorate] (9)--(10);
		\node [] at (7,-1.5) {$s_2$}; 
		
		\draw[thick] (3)--(1)--(4) (1)--(5);
		\draw[thick] (13)--(2)--(14) (2)--(15);	
	\end{tikzpicture}
	\caption*{$H_{s_1,s_2}$}

\begin{minipage}{0.45\linewidth}
	\centering
	\begin{tikzpicture}[scale=0.45]
		\node[draw, circle, inner sep=1.5pt] (1) at (2.5,4){}; 
		\node[draw, circle, inner sep=1.5pt] (2) at (6.5,4){};
		
		\node[draw, circle, inner sep=1.5pt] (3) at (-1,0){};
		\node[draw, circle, inner sep=1.5pt] (4) at (0,0){};
		\node[draw, circle, inner sep = 1.5pt] (5) at (3,0){};
		\node[draw, circle, inner sep = 0pt] at (1,0){};
		\node[draw, circle, inner sep = 0pt] at (1.5,0){};
		\node[draw, circle, inner sep = 0pt] at (2,0){};
		
		\node[draw, circle, inner sep=1.5pt] (13) at (6,0){};
		\node[draw, circle, inner sep=1.5pt] (14) at (7,0){};
		\node[draw, circle, inner sep = 1.5pt] (15) at (10,0){};
		\node[draw, circle, inner sep = 0pt] at (8,0){};
		\node[draw, circle, inner sep = 0pt] at (8.5,0){};
		\node[draw, circle, inner sep = 0pt] at (9,0){};
		
		\node[draw, circle, inner sep=1.5pt] (11) at (4.5,0){};
		\draw[thick] (1)--(11)--(2);
		
		\node[] (7) at (-1.5,0){};
		\node[] (8) at (3.5,0){};
		\draw [thick, decoration={brace, mirror, raise=0.3cm}, decorate] (7)--(8);
		\node [] at (1,-1.5) {$s_1$}; 
		
		\node[] (9) at (5.5,0){};
		\node[] (10) at (10.5,0){};
		\draw [thick, decoration={brace, mirror, raise=0.3cm}, decorate] (9)--(10);
		\node [] at (8,-1.5) {$s_2$}; 
		
		\draw[thick] (3)--(1)--(4) (1)--(5);
		\draw[thick] (13)--(2)--(14) (2)--(15);	
	\end{tikzpicture}
	\caption*{$M_{s_1,s_2}$}
\end{minipage}
\begin{minipage}{0.4\linewidth}
	\centering
	\begin{tikzpicture}[scale=0.5]
	\node[draw, circle, inner sep=1.5pt] (1) at (2.5,4){}; 
	\node[draw, circle, inner sep=1.5pt] (2) at (6.5,4){};
	
	\node[draw, circle, inner sep=1.5pt] (3) at (-1,0){};
	\node[draw, circle, inner sep=1.5pt] (4) at (0,0){};
	\node[draw, circle, inner sep = 1.5pt] (5) at (3,0){};
	\node[draw, circle, inner sep = 0pt] at (1,0){};
	\node[draw, circle, inner sep = 0pt] at (1.5,0){};
	\node[draw, circle, inner sep = 0pt] at (2,0){};
	
	\node[draw, circle, inner sep=1.5pt] (13) at (6,0){};
	\node[draw, circle, inner sep=1.5pt] (14) at (7,0){};
	\node[draw, circle, inner sep = 1.5pt] (15) at (10,0){};
	\node[draw, circle, inner sep = 0pt] at (8,0){};
	\node[draw, circle, inner sep = 0pt] at (8.5,0){};
	\node[draw, circle, inner sep = 0pt] at (9,0){};
	
	\node[draw, circle, inner sep=1.5pt] (11) at (4.5,0){};
	\draw[thick] (1)--(11)--(2);
	
	\node[] (7) at (-1.5,0){};
	\node[] (8) at (3.5,0){};
	\draw [thick, decoration={brace, mirror, raise=0.3cm}, decorate] (7)--(8);
	\node [] at (1,-1.5) {$s_1$}; 
	
	\node[] (9) at (5.5,0){};
	\node[] (10) at (10.5,0){};
	\draw [thick, decoration={brace, mirror, raise=0.3cm}, decorate] (9)--(10);
	\node [] at (8,-1.5) {$s_2$}; 
	
	\draw[thick] (3)--(1)--(4) (1)--(5);
	\draw[thick] (13)--(2)--(14) (2)--(15);	
	
	\node[draw, circle, inner sep=1.5pt] (12) at (11,0){};
	\end{tikzpicture}
	\caption*{ $M_{s_1,s_2}^*$}
\end{minipage}
\caption{Strongly acyclic bipartite graphs with one part of size $2$} \label{size2}
\end{figure}

Let  ${\cal M} =   \{H_{s_1,s_2}, M_{s_1,s_2}, M^*_{s_1,s_2}:  ~ s_1, s_2 \geq 0\} $ and  ${\cal H} = \{P_5', P_6, H_{3,4}^1, P_7\}$ be the set of graphs shown in Figures \ref{size2} and  \ref{acyclic-rest}. We denote  a cycle of length $i$ by $C_i$, a path on $i$ vertices by $P_i$,  and a complete bipartite graph with parts of sizes $s$ and $t$, $K_{s,t}$.

\begin{theorem}\label{characterization}
Let $H$ be a strongly acyclic bipartite graph.  Then either one of its parts has size $1$ or   one of its parts has size at most $3$ and  at least one of $H$ or its bipartite complement is in ${\cal M} \cup {\cal H}$. If $H$ is not a strongly acyclic bipartite graph, then $H$ or its bipartite complement contain $C_4$, $C_6$, or $C_8$.
\end{theorem}

\begin{proof}
Let $H$ be  a bipartite graph with the 
top part $U = \{u_1,\ldots,u_k\}$ and the bottom part $V= \{v_1,\ldots,v_l\}$, where $2 \le k \le l$. Denote by $H'$ the bipartite complement of $H$ and $d(u)$ a degree of a vertex $u$ in $H$.\\

Assume that $k=2$. Consider $u_1$, $u_2$ and their neighborhoods. We see that these neighborhoods share at most one vertex, otherwise we have a cycle of length four. The same holds for the bipartite complement of $H$.
Thus the only possibilities for $H$ or $H'$ are exactly graphs from ${\cal M}$ as shown in Figure \ref{size2}.\\

Now let $k\geq 3$.\\
Since $H$ and $H'$ are acyclic, the number of edges in $H$ and $H'$ is at most $|U|+|V|-1$, i.e., the total number of edges in these two graphs is at most $2(|U|+|V|-1)$. On the other hand, this number is $|U||V|$.  
We see however, that if $|U|,|V|= 4$, then $|U||V|> 2(|U|+|V|-1)$. Similarly, if $|U|=3$ and $|V|= 5$, we have that $|U||V|> 2(|U|+|V|-1)$. 
Thus, $k=|U|=3$ and $|V| \leq 4$.\\

Let $|U|=3$ and $|V|=3$. 
If there is a vertex from $U$ of degree $0$,  say $d(u_2) = 0$, then $d(u_1), d(u_3) \ge 2$, otherwise there is a $C_4$ in $H'$. Moreover we must  have $|N(u_1) \cap N(u_3)|\le 1$. Thus  $H = P_5'$. 
By considering $H'$, we can assume that  no vertex in $U$ has degree $3$.  Thus all vertices of $H$ have degrees $1$ or $2$ and the number of edges is $3$, $4$, $5$, or $6$. Since $H$ is strongly acyclic, there could be at most $5$ and at least $4$  edges. 
So, up to bipartite complementation we can assume that there are $4$ edges in $H$ with respective degrees $1$, $1$, and $2$ in both parts. This is only possible when $H$ is a disjoint union of $K_2$ and $P_4$, whose bipartite complement is $P_6$.\\

Let $|U|=3$ and  $|V| = 4$.
Assume there is a vertex $u \in U$ with $d(u) = 4$. Then any two other vertices in $U$ have $3$ non-neighbors in $N(u)$ each, and thus, at least two common non-neighbors, resulting in $C_4$ in $H'$.  By considering $H'$, we see that there are no vertices of $U$ of degree $0$.  I.e., the degrees of vertices from $U$ could be $1, 2$, or $3$.  Since $H$ is strongly acyclic, the number of edges is at most $6$ and at least $12-6=6$.  So, $H$ has $6$ edges, and degrees of vertices in $U$ are $1,2,3$ or $2,2,2$.
In case of degrees $1,2,3$ we see that the neighborhoods of degree $2$ and $3$ vertices intersect in exactly one vertex. The vertex of degree $1$ must  be adjacent to a neighbor of degree $3$ vertex that is not adjacent to degree $2$ vertex, otherwise there is a $C_4$ in the bipartite complement of $H$. Thus we have that $H=H_{3,4}^1$.
If the degrees of vertices in $U$ are $2,2,2$, then the only option is $P_7$.\\

We only need to show that any bipartite graph $H$  that is not strongly acyclic, contains $C_4, C_6$, or $C_8$ in it or its bipartite complement. If $H$ has one part of size at most $4$, we are done, since any cycle in $H$ or $H'$ has length at most $8$.
If $H$ has both parts of sizes at least $5$, one can easily verify that either $H$ or $H'$ contains a $C_4$.  \end{proof}

\section{Proof of Theorem \ref{main}}\label{proof}

Let $H_s = H_{s,s}$, $M_s=M_{s,s}$ and $M^*_s = M^*_{s,s}$.  In the following results we observe that if $H$ is an induced bipartite subgraph of $K$ respecting sides, then $\tih(n, H)\geq \tih(n, K)$. We omit ceilings and floors where it is not essential.

\begin{lemma}\label{maxdegree}
Let $G$ be a bipartite graph with parts $U$ and $V$ of sizes $n$ each and  degrees of vertices from $U$  less that $s$. Then $\widetilde{\alpha}(G) \geq n/s$.
\end{lemma}

\begin{proof}
Let $U'$ be a subset of $n/s$ vertices of $U$. Then $|N(U')|\leq (s-1)n/s = n - n/s$.  Let $V' = V- N(U')$, we have $|V'|\geq n/s$ and $(U', V')$ form a co-biclique.
\end{proof}

\subsection{Forbidden $H_s$}

\begin{lemma}\label{1} Let $s_1\geq s_2> 0$.  Then $\tilde h(n, H_{s_1,s_2}) \ge  \tih(n, H_{s_1}) \geq \frac{n}{2s_1}.$
\end{lemma}

\begin{proof}[Proof of Lemma \ref{1}]
Let $H=H_s$, for $s = \max\{s_1,s_2\}$. Let $G$ be an $H$-free  bipartite graph  with top part $U$ and bottom part $V$, both of size $n$.  We will show that $\tilde h(G) \ge \frac n{2s} $. Let $\{u_1,\ldots,u_n\}$ be an ordering of the vertices of $U$, s.t. $d(u_i) \le d(u_j)$ if $i < j$. Since $H$ is isomorphic to its bipartite complement, we can assume that  $d(u_{n/2}) \le \frac n2$.
	Assume first that  there is an $i < \frac n2$ with $|N(u_i) \setminus N(u_{n/2})| \ge s$. Then we have a set $V'$ of $s$ vertices, $V' \subseteq  N(u_i)\setminus N(u_{n/2})$, and since $|N(u_i)| \le |N(u_{n/2})|$, we have also a set of $s$ vertices $V''$,  $V''\subseteq  N(u_{n/2})\setminus N(u_i) $. But then $\{u_i, u_{n/2}\}$ and $V'\cup V''$ induce $H$.  Let $Y = V \setminus N(u_{n/2})$. We have $|Y| \ge \frac n2$ and by the above argument, we have $|Y \cap N(u_i)| \le s-1$, for all $i \in [\frac n2]$. Applying Lemma \ref{maxdegree} to a subgraph of $G$ induced by $Y$ and  $\{u_1, \ldots, u_{[\frac n2]}\}$, we get $\widetilde{\alpha}(G) \ge \frac{n}{2s}$.
\end{proof}

\begin{remark}
In the case where $s_2 = 0$, we can even show that $\tilde h(n, H_{s,0}) \ge \frac{n}{2s-1}.$
\end{remark}

\subsection{Forbidden $M_s$ or  $M_s^*$}

In this section we need an auxiliary lemma about rooted trees. 
We call two vertex disjoint subforests of a rooted tree \textit{independent} if no vertex in one forest is an ancestor of a vertex in the other forest. 
We say that a maximal path rooted at the root of a rooted tree $T$ with inner vertices of degree $2$ in $T$ is a \textit{handle} of $T$, denote its vertex set $H(T)$.

\begin{figure}[H]\label{t}
	\centering
	\includegraphics[width=0.4\linewidth]{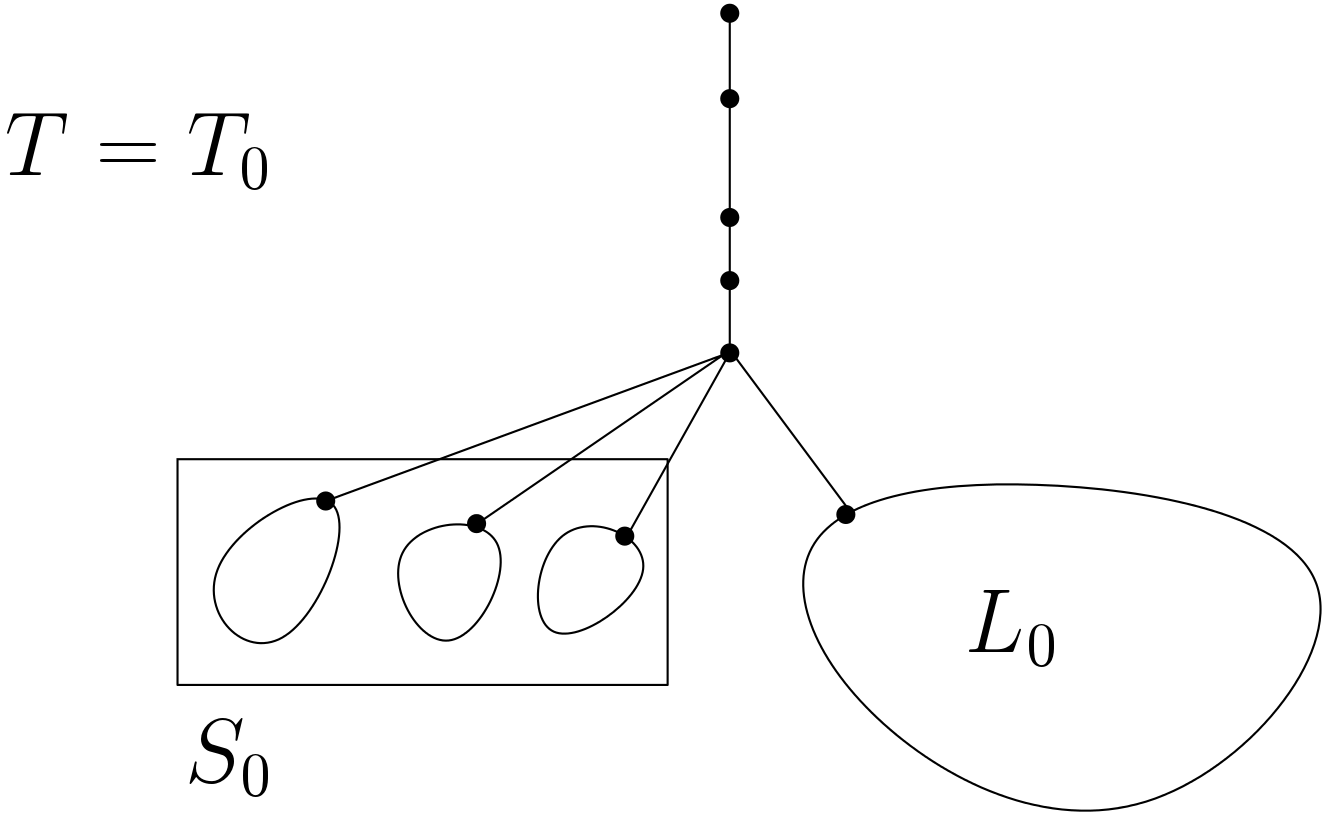}\qquad\\
	\vskip 0.9cm
	\includegraphics[width=0.4\linewidth]{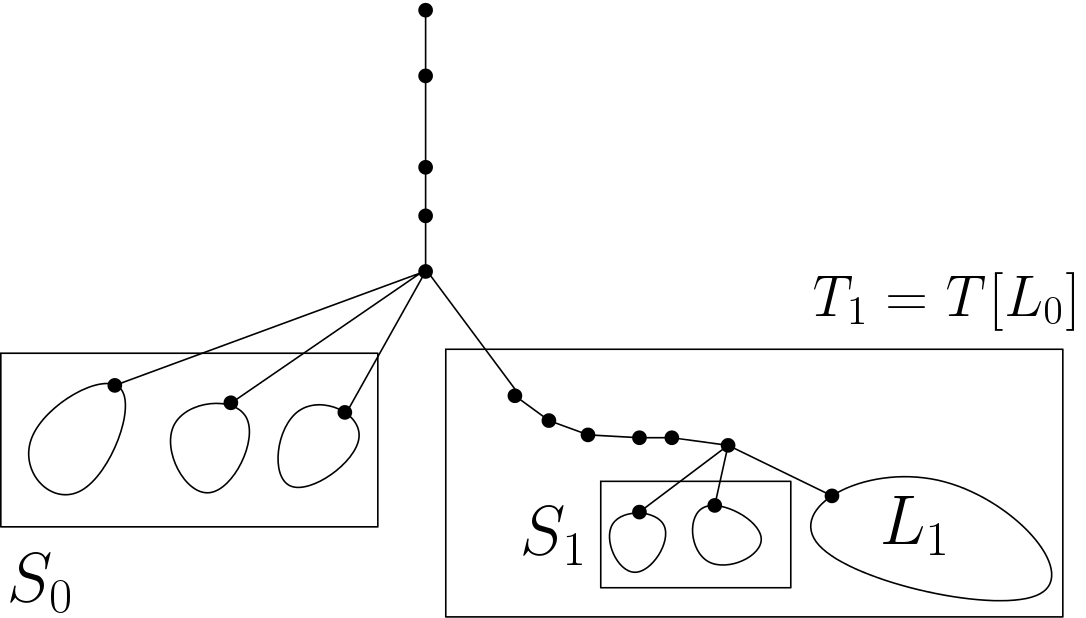}\qquad\\
	\vskip 0.9cm
	\includegraphics[width=0.6\linewidth]{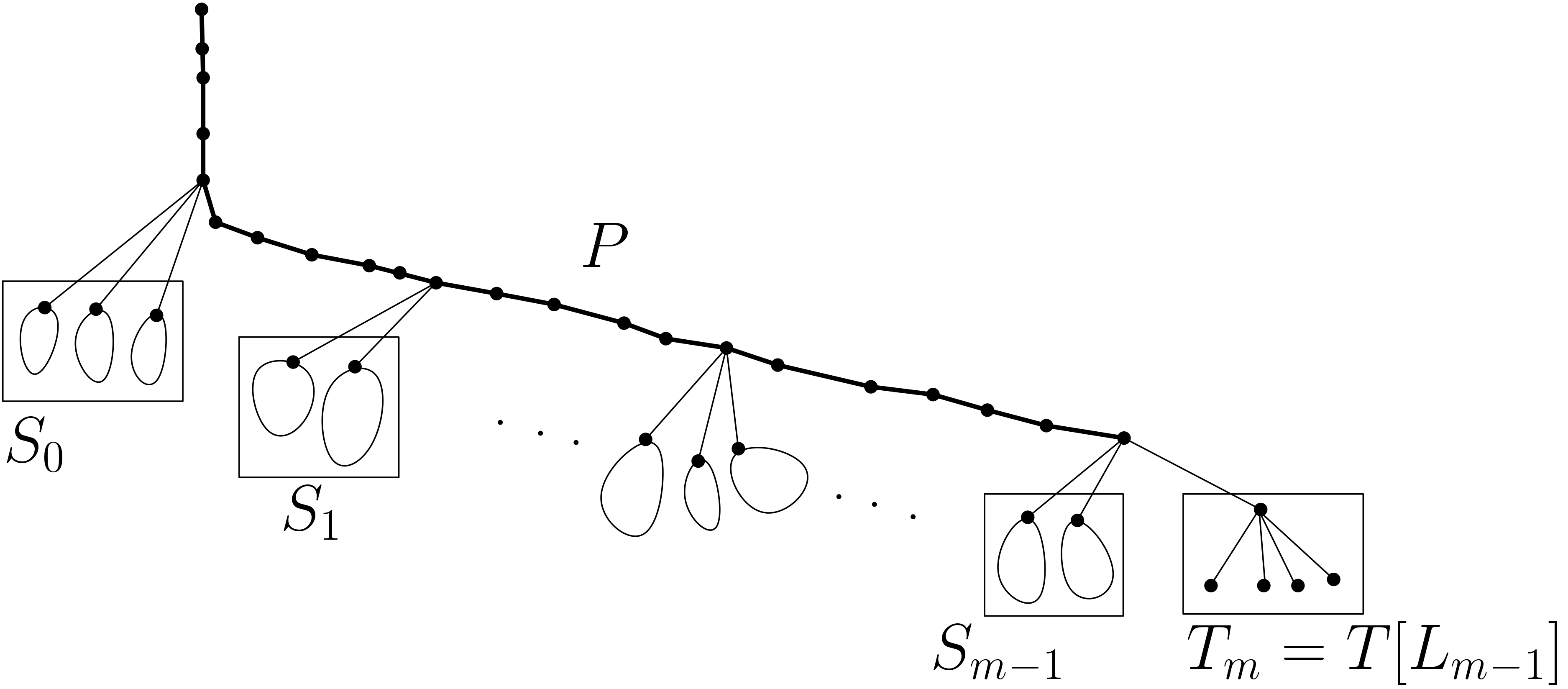}
	\caption{Illustration of Lemma \ref{trees}}
\end{figure}

\begin{lemma}\label{trees}
Any rooted tree on $n$ vertices has either height at least $\frac n4$, or it contains two independent subforests on at least $\frac n4$ vertices each.
\end{lemma}

\begin{proof}
We shall define vertex sets $S_i, L_i$ as follows. Let $T_0=T$.
Let $L_0$ be the vertex set of a largest component of $T_0-H(T_0)$, let $S_0$ be the set of vertices in all other components of $T_0-H(T_0)$. Assume that $T_0, \ldots, T_{i-1}$ be defined, as well as $S_0, \ldots, S_{i-1}$ and $L_0, \ldots, L_{i-1}$.
Let $T_i$ be the tree induced by $L_{i-1}$. Let $L_i$ be the set of vertices of a largest component of $T_i- H(T_i)$, 
let $S_i$ be the set of vertices in all other components of $T_i- H(T_i)$. We stop with $T_m$ being a star. Let $S_m$ be the set of leaves in $T_m$.   We have that $T$ is spanned by $S_0, \ldots, S_{m}$ and a path $P$ built out of handles. See Figure 3 for an illustration.
  We see that $S_0 \cup  \cdots \cup  S_m$ span a forest with pairwise independent components in $T$.
Let $t_i$, $i=0, \ldots, \ell$ be the sizes of components in this forest.  We would like to group these components into two balanced parts.  Recall that $t_i\leq n/2$ for each $i$.
If $|P|\geq n/4$, we are done. Assume that $|P|<n/4$.
Thus $t_0+\cdots + t_\ell = n-|P|>3n/4$.
Consider a partition $\{0, \ldots, \ell\}= I\cup J$ such that $t_I:= \sum_{i\in I } t_i$ is as close to  $t_J:= \sum_{i\in J} t_i$ as possible. Let $t_I\geq t_J$. If $t_I, t_J\geq n/4$, we are done.  If not, then $t_J<n/4$, $t_I>n/2$. 
Then in particular, $t_i>n/4$ for each $i\in I$, otherwise we would move this $i$ from $I$ to $J$ and create a more balanced partition.
In addition, we have that $I$ consists of one element, say $I=\{i\}$, otherwise we can again move an some $i$ from $I$ to $J$. This, however, contradicts the fact that each $t_i\leq n/2$.
\end{proof}

\begin{lemma}\label{2} Let $s_1\geq s_2> 0$. Then $ \tilde h(n, M_{s_1,s_2}) \ge  \tih(n, M_{s_1} )\geq \frac{n}{54s_1}$ and 
	$ \tilde h(n, M^*_{s_1,s_2}) \ge \tih(n, M^*_{s_1}) \geq \frac{n}{108s_1}$.
\end{lemma}

%
%

\begin{proof} Let $s= s_1$, $s\geq 1$.
Let $G''\subseteq K_{n,n}$ have partite sets $U''$ and $V$ and such that $G''$ has no induced copy of $M_s$ with  the smaller partite set in $U''$. Assume that $\tilde{h}(G'')< n/(25s)$. \\

{\it Proof outline:} We shall first delete a few vertices of small degrees  (at most $6s$) and some other sets of vertices so that the remaining ones belong to blobs such that any two vertices in the same blob have degrees different by at most $2s$  and any two vertices from different blobs have degrees different by at least $2s$.  We shall call the resulting graph $G$ and its parts $U$ and $V$.  We introduce an auxiliary graph $I$ on $U$ whose edges correspond to two vertices with intersecting neighborhoods in $G$   and show that this auxiliary graph has a very special structure, i.e., formed of vertex-disjoint cliques that are  pairwise either completely adjacent or completely disjoint. This gives rise to the second auxiliary graph $J$ for which we show that it is a tree closure of some tree.\\

Let $S$ be the set of vertices of $G''$ from $U''$ of degree at most $6s$.  Assume $|S|\geq n/(24s)$, let $S'\subseteq S$, $|S'| = n/(24s)$. Then if $V'= N(S')$, $|V'|\leq n/4$. Thus  $(V-V', S')$ form a co-biclique with parts of sizes at least $n/(24s)$. This contradicts our assumption that  $\tilde{h}(G'')< n/(25s)$. Thus $|S|\leq n/(24s)$.\\

Let $G'= G''- S$, where $S$ is the set of vertices of $U''$ of degree at most $6s$. Let $U' = U''-S$.
Let $U' = U_1\cup \cdots \cup U_{n/(2s)}$, where $U_i =\{ v\in U': ~ (i-1) 2s\leq d(v) \leq i 2s\}$. Since degree of vertices in $U'$ are at least $6s$, $U_1 $ and $U_2$ are empty.
Let $U_e= U_4\cup U_6\cup \cdots$, $U_o =  U_3\cup U_5 \cup \cdots $. Assume without loss of generality that 
$|U_e|\geq |U'|/2$.  Consider $G = G'[U_e, V]$. Let $U= U_e$. Then $|U|\geq (n- n/(24s))/2 = n/2 - n/(48 s).$\\

Introduce an {\it auxiliary graph } $I$ with vertex set $U$ and two vertices  adjacent iff their neighborhoods in $G$ intersect.
We shall show that each component of  $I$ is a {\it closure of a rooted tree} or just a {\it tree closure}, i.e., a graph obtained from a rooted tree by adding, for each vertex $v$,  all edges between  $v$ and each of  its ancestors. Here an ancestor is  a vertex on a path from $v$ to the root.\\

\noindent
{\it Claim 0.}  Let $x, y, z\in U$, $xy, yz \in E(I), xz\not\in E(I)$. Then $d(x)+d(z) < d(y)+2s$.\\

Assume otherwise, then w.l.o.g. $|N(x)\setminus N(y)|\geq s$. Then $|N(y)\setminus N(x)|<s$, implying  $|N(y)\setminus N(z)|>s$, that in turn implies that $|N(z)\setminus N(y)| <s$. Since $N(z)\cap N(y) \subseteq N(y)\setminus N(x)$, $|N(z)\cap N(y)|<s$. Thus $|N(z)|< 2s$, a contradiction.\\

\noindent
{\it Claim 1.}  For any $i = 2, \ldots,  n/(4s)$,  $I[U_{2i}|$ is a pairwise vertex disjoint union of cliques. Call the family of these cliques $\K_i$, refer to  this family as an $i^{\rm th}$ {\it blob}.\\

Assume that $I[U_{2i}|$ is not a disjoint union of cliques. Then there are three vertices $x, y, z\in U_{2i}$ such that $xy, yz \in E(I), xz\not\in E(I)$. We have that the degrees of $x$, $y$, and $z$ differ by at most $2s$. Thus $d(x)+d(z)\geq d(x) + d(x) -2s$.
On the other hand $d(y) + 2s\leq d(x) + 2s + 2s$.  From Claim 0 we have  $d(x)+d(z)<d(y)+2s$. This implies that $2d(x) - 2s <d(x) +4s$, i.e., 
$d(x)<6s$, a contradiction.\\

The following two claims also follow from Claim 0 similarly. We use the fact that by definition, the degrees of vertices in $G$ from two distinct blobs differ by at least $2s$.\\

\noindent
{\it Claim 2.}  For any $K\in \K_i$ and $K'\in \K_j$, $i\neq  j$, the bipartite graph with parts $V(K), V(K')$ is either complete or empty.\\

If not,  there are $x,y $,  and  $z$ such that $xy, yz \in E(I), xz\not\in E(I)$, where  either ($z\in V(K)$ and $x,y\in V(K')$) or  
($z\in V(K')$ and $x,y\in V(K)$). 
Then we see that $|d(x)-d(y)|\leq 2s$, and $d(x), d(y), d(z)> 6s$. 
Thus $d(x)+d(z) > 6s+ d(x)\geq 4s + d(y) $, a contradiction to Claim 0. \\

Let $J$ be a graph with vertex set $\K_1\cup \cdots \cup \K_{n/(4s)}$ with two cliques adjacent iff there is an edge between them in $I$. \\

\noindent
{\it Claim 3.}  If $K\in \K_i,  K'\in \K_j, K''\in \K_k$, and $KK', K'K'' \in E(J)$, and  $KK''\notin E(J)$, then $j\geq i$ and $j\geq k$.\\

Assume $x\in V(K), y\in V(K'), z\in V(K'')$, $xy, yz\in E(I)$, $xz\not\in E(I)$. 
Then by Claim 0 $d(x)+d(z) < d(y)+2s$.  Thus $d(y) \geq d(x)-2s$ and $d(y)\geq d(z)-2s$. 
Thus $j\geq i$ and $j\geq k$.\\

\noindent
{\it Claim 4.} Each induced connected subgraph  $F$  of $J$ contains a vertex adjacent to all other vertices in $F$. \\

We shall prove this by induction on the order of $F$ with a trivial basis of a one-vertex graph. 
Assume that the statement is true for all connected subgraphs with order less than $|V(F)|$.
Let $Z_m$ be  a vertex of $F$ that belongs to the highest indexed blob. Let $F'$ be a component of $F-\{Z_m\}$.
Let $Z_q$ be a vertex of $F'$ that belongs to the highest indexed blob. 
By induction $Z_q$ is adjacent to all other vertices of $F'$.
Since $F$ is connected, $Z_m$ is adjacent to some vertex $Z_j$ of $F'$.
If $Z_m$ is not adjacent to $Z_{q}$, then either $F$ is disconnected (in case $|V(F')|=1$) or  the vertices 
$Z_j, Z_m, Z_{q}$ contradict Claim 3. So, $Z_{q}Z_m\in E(J)$.  If $Z_m$ is not adjacent to some vertex $Z_i$ of $F'$, then again $Z_i, Z_{q}, Z_m$ contradict Claim 3. Thus $Z_m$ is adjacent to all  vertices of $F'$. Since this holds for each component of $F-\{Z_m\}$, we see that $Z_m$ is adjacent to all other vertices of $F$.  Claim 4. implies the following.\\

\noindent
{\it Claim 5.}  Each component of $J$ and thus of $I$  is a tree closure for some tree.\\

\noindent
{\it Claim 6.}  If a component of $I$ has  vertex set $Q$ and an underlying tree with two independent subforests on vertex sets $Q_1, Q_2$ of sizes at least $|Q|/4$ each, then in $G$  there is a co-biclique with parts in $Q, N(Q)$ of sizes at least $|Q|/4, |N(Q)|/2$ respectively.\\

Since $Q_1$ and $Q_2$ induce independent forests in an underlying tree of a component of $I$, we have that $N(Q_1)\cap N(Q_2)=\emptyset$. Thus, without loss of generality, $|N(Q_1)|\geq |N(Q)|/2$. Then $Q_2$ and $N(Q_1)$ form a desired co-biclique.\\

\noindent
{\it Claim 7.} Let  $X$ be a subset of vertices from $U$ that induces pairwise disjoint union of cliques in $I$. If $|X|= cn$ for a constant $c \le 1$,  then $\tilde{h}(G) \geq cn/(6s)$.\\

Assume first that there is a set $S$  that induces a clique in $I[X]$ of size at least $cn/3$. We can assume that if $V_S= N_{G}(S)$ then $|V_S|\geq cn/3$, otherwise $(S,V-V_S)$ induce a co-biclique with parts of sizes $cn/3$ and $n - cn/3$, that is at least $ cn/3$ for $c\le 1$. Thus $G[S, V_S]$ is a bipartite graph with parts of sizes at least $cn/3$, and this graph has no induced $2K_{1, s}$, otherwise these stars and  a common neighbor of their centers induce $M_s$. Thus by Lemma \ref{1}, $\tilde{h}(G) \geq cn/ (6s)$.

Now, we are left with the case that all cliques in $I[X]$ have size at most $cn/3$. 
Split the cliques of $I[X]$ in two groups of total size at least $cn/3$, let the vertex sets of these groups be $S'$ and $S''$.
Let w.l.o.g.  $|N(S')|<|N(S'')|$ in $G'$. Then $(V-N(S'), S'')$ induces a co-biclique with parts at least $cn/3$.\\

\noindent
{\it  Final argument: }
Consider a component of $I$ on a vertex set $Q$. By Lemma \ref{trees}, it either has a clique of size $|Q|/4$ (when the underlying tree has respective height) or gives a co-biclique in $G$ with parts of sizes $|Q|/4$ and $|N(Q)|/2$. \\

\noindent
{\it Case 1.} At least a $\frac{6s}{6s+1}$-fraction of the vertices in $U$  is spanned by the components of $I$  with large cliques, i.e., giving a subset of $\frac{|U|}4 \frac{6s}{6s+1} $ vertices that is a union of cliques in $I$. Thus by Claim 7. 
$$\tilde{h}(G) \geq  \frac{1}{6s}\frac{|U|}{4}\frac{6s}{6s+1}\geq \frac{1}{24s+4}\left(\frac{n}{2}-\frac{n}{48s}\right)\geq \frac{n}{48s+8} - \frac{n}{1152s^2+196s}.$$

\noindent
{\it Case 2.}  At least a $\frac{1}{6s+1}$-fraction of the vertices of $U$ is spanned by the components of $I$ that give large co-bicliques.  
Let $\cal{Q}$ be the family  of vertex sets of respective components. So, for each $Q\in \cal{Q}$,  we have a subset $Q_1$ of $Q$, $|Q_1|\geq |Q|/4$ that forms a co-clique with $Q''\subseteq N(Q)$, $|Q''|\geq |N(Q)|/2$. 
Then $$U^*= \bigcup_{Q\in \cal{Q}}  Q_1  ~~\mbox{  and }~~ V^*= \bigcup_{Q\in \cal{Q}}  Q'' \cup \left(V\setminus N\left(\bigcup_{Q\in \cal{Q}} Q\right)\right)$$
  form partite sets of  a co-biclique. Note that $|U^*|\geq \frac {|U|}4\frac{1}{6s+1}$, $|V^*|\geq |V|/2\geq n/2$.
Thus 
$$\tilde{h}(G) \geq \frac{n}{48s+8} - \frac{n}{1152s^2+196s}.$$ 

We have that the lower bounds in Case 1 and Case 2  are larger than $\frac{n}{54s}$ for $s \ge 2$. For $s=1$, these bounds are at least $\frac{n}{59s}$, however we can give an explicit argument for a bound of $n/3$ for $s=1$.   Note that one can definitely improve on the bound  $\frac{n}{54s}$.  This concludes the proof of the Lemma for $M_s$. \\

Now, consider an $M_{s}^*$-free bipartite graph $G$ with top part $U$ and bottom part $V$, both of size $n$. 
If there is a vertex $v \in V$ of degree $d(v) \leq  n/2$, then the graph $G[U-N(v), V\setminus\{v\}]$ is $M_s$-free.
 Thus, by the previous result on $M_s$,  we have $\tilde h(G) \ge n/({2\cdot 54s})$.  If $V$ does not contain a vertex of degree at most  $n/2$, consider the bipartite complement $G'$ of $G$. Since $M^*_s$ is (bipartite) self-complementary, $G'$ does not contain $M^*_s$ either, but we have a vertex $v \in V$ with $d_{G'}(v) \leq   \frac n2$ and thus, we can apply the same argument.  \end{proof}

\begin{proof}[Proof of Theorem \ref{main}]
Consider a strongly acyclic graph $H$ such that neither $H$ nor $H'$ is in ${\cal H}$. Then by Theorem \ref{characterization}
$H$ has one part of size $1$ or $2$. If one part is of size $1$,  i.e., this part consists of a vertex $v$ that has $s$ neighbors and $t$ non-neighbors for some $s$ and $t$. Thus if $G$ is $H$-free, each vertex in the top part has either degree at most $s-1$ or degree at most $t-1$ in a bipartite complement. Assume without loss of generality that at least half of the vertices in the top part have degree at most $s-1$.  Applying Lemma \ref{maxdegree}  to these vertices in the top part and the bottom part, we see that $H$ has  a bi-coclique with parts of sizes at least $n/2s$.
If $H$ has one part of size $2$ and another part of size at least $2$, then by Theorem \ref{characterization}, $H$ is $H_{s_1,s_2} $, $M_{s_1,s_2} $ or $M'_{s_1,s_2} $, with $s_1, s_2 \ge 0$, or their bipartite complement.  
Then the result follows from Lemmas \ref{1} and \ref{2}.
\end{proof}

\section{Tight bounds for all strongly acyclic graphs with two vertices in each part}\label{small}
We consider strongly acyclic bipartite graphs with each part of size $2$.  These are exactly $2K_2$, $(2K_2)'$,  $P_4$, and $H_4$, 
where $H_4$ is such a graph with exactly two adjacent edges, $2K_2$ has two disjoint edges.
We shall give bounds for $\tih(H)$ for each of these graphs. Recall that $\tih(H)=\tih(H')$, where $H'$ is a bipartite complement of $H$.

\begin{lemma}
Let $G$ be a bipartite $P_4$-free graph with $n$ vertices in each part. Then $\tih(G)  \geq n/3$. This bound is tight for $n\equiv 0 \pmod 3$.
\end{lemma}

\begin{proof}
Let $G$ have partite sets $U$ and $V$. It is easy to see that  $G$ is a pairwise vertex disjoint union of bicliques.  Let, for some index set $I$,  these bicliques have partite sets $U_i$ and $V_i$ of sizes $a_i$, $b_i$, respectively, $U_i\subseteq U$, $i\in I$. Observe first that $\min\{a_i, b_i\}< n/3$, for each $i\in I$, otherwise $i^{\rm th}$ biclique gives us $\tilde{\omega}(G)>n/3$.
Moreover $a_i<n/3$ and $b_i<n/3$,  for each $i\in I$ since otherwise  a cobiclique with parts $U_i, V-V_i$ or $V_i, U-U_i$ has parts of sizes greater than $n/3$.

Let $I'$ be the set of indices, so that $a_i\leq b_i$, $i\in I'$. Let $I''=I\setminus I'$.
Let $U' = \cup_{i\in I'}  U_i$,  $V' = \cup _{i\in I'} V_i$, $U''= U-U'$, $V''= V-V'$,  $a' =|U'|$, $a'' = |U''|$, $b' =|V'|$, $b''=|V''|$.
Consider a cobiclique with parts $V', U''$. We can assume that either $b'$ or $a''$ is less than $n/3$, say $b'<n/3$. 
Then $a'<n/3$ since for each $i\in I'$, $a_i\leq b_i$. Thus $a''>2n/3$.

Consider a minimal subset  $I'''\subseteq I''$ such that $U''' = \cup _{i\in I'''} U_i$ has size $a''' >n/3$.  Then  $a'''< 2n/3$ otherwise for any $i\in I'''$, 
$|U''' - U_i| > 2n/3 - n/3 = n/3$. In particular, we could have taken $I'''-\{i\}$ instead of $I'''$, contradicting its minimality. 
Thus $V''' = \cup _{i\in I'''} V_i$ has size less than $2n/3$. This implies that  $U'''$ and $V-V'''$ form a co-biclique with each part of size at least $n/3$.

Note, that the bound shown is best possible, by the following $P_4$-free construction for every natural $n$
with $n \equiv  0 \pmod 3$.  Take $G$ to be  the disjoint union of $3$ complete bipartite graphs $K_{n/3,n/3}$.
Clearly this is $P_4$-free, and $\tih(G) = n/3$. 
\end{proof}

\begin{lemma}
Let $G$ be a $2K_2$-free  bipartite graph with $n$ vertices in each part. Then $\tih(G)  \geq \ceil{\frac n2}$. This bound is tight. 
\end{lemma}  

\begin{proof}
	Let $G$ be bipartite $2K_2$-free with parts $U, V$ of size $n$ each. Then we have for any vertices 
$ u,u' \in U$   that $N(u) \subseteq N(u')$ or  $N(u') \subseteq N(u)$.  Thus there is a total ordering $\{u_1,\ldots, u_n\}$ of the vertices in $U$ where $i < j$  if and only if  $N(u_j) \subseteq N(u_i)$. Consider the two subgraphs $G_1 = G[U_1 \cup V_1]$, $G_2 = G[U_2 \cup V_2]$, with 
$U_1 = \{u_1,\ldots, u_{\ceil{\frac n2}} \},$ $~ V_1 = N(u_{\ceil{\frac n2}}), $ $ ~
		U_2 = \{u_{\ceil{\frac n2}}, \ldots, u_n \},$ $~ V_2 = V \setminus N(u_{\ceil{\frac n2}}).$
	
%
%
	By our vertex ordering, we have that $N(u_{\ceil{\frac n2}}) \subseteq N(u_i)$, $1\le i < \ceil{\frac n2}$, and thus, $G_1$ is a biclique.
	On the other hand,  $V \setminus N(u_{\ceil{\frac n2}}) \subseteq V\setminus N(u_i)$, $\ceil{\frac n2} < i \le n$, and thus, $G_2$ is a co-biclique. We know  that $|U_1| = |U_2| = \ceil{\frac n2}$. Since $|V_1| + |V_2| = n$, one of them has to have size at least $\ceil{\frac n2}$, which gives us $\max\{\widetilde{\omega}(G), \widetilde{\alpha}(G) \} \geq  \ceil{\frac n2} $. \\
	
	Note, that $\ceil{\frac n2}$ is best possible. Consider a bipartite graph $G$ that is a union of  a complete bipartite graph $K_{\ceil{n/2}, n}$ and add $\floor {\frac n2}$ isolated vertices added to the smaller part. Both parts have size $n$, we have $\tih(G) = \ceil{\frac n2}$ and $G$ is $2K_2$-free.
\end{proof}

\begin{lemma}
Let $G$ be an $H_4$-free  bipartite graph with $n$ vertices in each part. Then $\tih(G)  \geq \lfloor 2n/5 \rfloor $. This bound is tight for $n\equiv 0 \pmod 5$.
\end{lemma}

\begin{proof}
Let $G$ have top part $U$ and bottom part $V$ of size $n$ each, assume $n$ is divisible by $5$. Denote by $G'$ the bipartite complement of $G$.  First observe, that $|N(u) \setminus N(u')| \le 1$,  for any  $ u, u' \in U$. \\

\noindent
{\it Claim}: There is  a set $X \subset V$ such that for a graph $Q$ that is either $G$ or its bipartite complement, for all $u \in U$ we have $N_{Q} (u)  - v \subseteq X \subset N_Q(u)$ for some $v=v(u) \in V$. I.e., the neighborhoods of the vertices from $U$ form a sunflower set system with petals of sizes at most one.\\

	To prove the Claim, consider a vertex  $u' \in U$ of largest degree.  If  $N(u) \subseteq  N(u')$ for each $u \in U$, let $X= V \setminus N(u')$. Then $X$ satisfies the conditions of the Claim with $Q=G'$. So, we assume that there is a vertex $v''\in N(u'')\setminus N(u')$ for some vertex $u''\in U$.   In particular $d(u'')=d(u')$. Let $V'= N(u')\cap N(u'')$, $V'' = N(u')\cup N(u'')$.  We see that for each $u\in U$, either $V'\subseteq N(u)$ or $N(u)\subseteq V''$.

	 If for all vertices $u\in U$, $N(u)\subseteq V''$, then $X= V-V''$ satisfies the conditions of the Claim with $Q=G'$. If for all $u\in U$, $V'\subseteq N(u)$, then the Claim is satisfied with $X=V'$ and $Q=G$. 
If there are  $u_1, u_2$ such that $u_1, u_2\not\in \{u', u''\}$, $V'\not\subseteq N(u_1)\subseteq V''$ and $V'\subseteq N(u_2)\not\subseteq V''$, we see that $u_1, u_2, v_1, v_2$ form a copy of $H_4$, where $v_1, v_2 $ is the symmetric difference of $N(u')$ and $N(u'')$. This proves the Claim. \\

	\noindent
	Now that we proved the Claim, it remains to find a large biclique or (co-)biclique. Assume without  loss of generality that $Q=G$ in the Claim.
	If $|X| \geq 2n/5$,  then $(U,X)$ induces a biclique with parts of sizes at least  $2n/5$.  Assume that  $|X| < 2n/5$. Let $Y = V \setminus X$.  
	We see that $(U, Y)$ induces a pairwise disjoint union of stars in $G$ with centers in $Y$. Let $s= |Y|$, note that $s\geq 3n/5$.
	Let $Y = \{y_1,\ldots, y_{s}\}$ such that  $d(y_i) \le d(y_{i+1}) $, $i=1,\ldots, s-1$.
	Let $Y_1 = \{y_1, \ldots, y_{2n/5}\}$. If $|N(Y_1)|\leq 3n/5$, then $(Y_1, U\setminus N(Y_1))$ induces a co-biclique with parts of sizes at least $2n/5$. If $|N(Y_1)|>3n/5$, then $d(y_i)\geq 2$ for all $i>2n/5$. Thus $|N(Y)| \geq |N(Y_1)| + 2|Y-Y_1|> 3n/5 + 2\cdot (3n/5-2n/5)= n$, a contradiction since $N(Y)\subseteq U$.\\

%
%
%
%
%
%
%

	To show that the bound is tight,  construct the following graph $G$ with parts $U$ and $V$ of sizes $n$, $n \equiv 0 \pmod 5$.
	Let $U=U_1\cup U_2$ where $|U_1|=2n/5$ and $|U_2|= 3n/5$.  Let $V= V_1\cup V_2\cup V_3$, where $|V_1|=|V_2|=2n/5$ and 
	$|V_3| = n/5$. Let $G$ have all edges between $V_1$ and $U$, form a perfect matching between $U_1$ and $V_2$, and form a perfect matching between $U_2$ and $V_2\cup V_3$. Note that if $G$ has a copy of $H_4$, this copy has a vertex $u$ of degree $2$ in $U$. Thus, this copy must have a neighbor of $u$ in $V_1$, that in turn is adjacent to all of $U$ and thus could not have degree $1$ in a copy of $H$. Thus $G$ is $H_4$-free.  In addition, we see that $\tih(G) = 2n/5$.
\end{proof}

\section{General bounds}\label{general}

In this section we work out known arguments for completeness.

\begin{theorem}
Let $H$ be a bipartite graph that is not strongly acyclic. Then there is an $\epsilon >0$ such that for each  sufficiently large $n$, 
$\tih(n,H) \leq n^{1-\epsilon}$.  Moreover, if $H$ or its bipartite complement contains $C_4$, $C_6$, or $C_8$, 
then $\epsilon$ could be taken any positive real strictly less than $1/3,  1/6, $ or $1/16$, respectively.
\end{theorem}

\begin{proof}
First, recall from Theorem \ref{characterization}, that if $H$ is not strongly acyclic, then $H$ or its bipartite complement contains $C_4, C_6$, or $C_8$.  In case of $C_4$, we know by a result of Caro and Rousseau  \cite{CR} that there is a bipartite graph $G$  with parts of size $n$ each that does not contain $C_4$ and such that $\talp(G)=O(n^{2/3})$. This result was shown using Lov\'asz Local Lemma that we abbreviate as  LLL.
The LLL tells us that if there are bad events $A_i, \ldots$ and positive numbers $x_i, \ldots$ associated with these events  such that 
$Prob(A_i) \leq (1-x_i) \prod_{j\sim i} x_j$, where $i\sim j$ iff $A_i$ is adjacent to  $A_j$ in the dependency graph, then  with positive probability none of the bad events happen. \\

We use the same approach to randomly create respective graphs with no $C_6$ and with no $C_8$ and not having large co-cliques.
Consider $K_{n,n}$ and color each edge red with probability $p$ and blue with probability $(1-p)$. 
For a specific $C_6$, we say that there is a red bad event if this $C_6$ is red. Similarly, for a specific $K_{t,t}$, we say that there is  a blue bad event if this $K_{t,t}$ is blue. 
We shall use LLL to prove that with positive probability there are no bad events. We say  A "depends" on B when an event A is adjacent to B in the dependency graph.

\begin{itemize}
\item[-]
The probability of a red bad event is $p^6$. The probability of a blue bad event is $(1-p)^{t^2}$.
\item[-]
Each red bad event depends on at most $6 \binom{n}{2}^2 4 \leq 6n^4$ other red bad events (this corresponds to the number of $C_6$'s sharing an edge with a given $C_6$).
\item[-]
Each red bad event depends on at most $6 \binom{n}{t-1}^2 \leq n^{2t}$ blue bad events (this corresponds to the number of $K_{t,t}$'s sharing an edge with a given $C_6$).
\item[-] 
Each blue bad event depends on $t^2\binom{n}{2}^2 4\leq t^2n^4$ red bad events (this corresponds to the number of $C_6$'s sharing an edge with a given $K_{t,t}$).
\item[-]
Each blue bad event depends on $t^2\binom{n}{t-1}^2 4\leq  n^{2t}$ blue bad events (this corresponds to the number of $K_{t,t}$'s sharing an edge with a given $K_{t,t}$).
\end{itemize}

\noindent
Since here we have bad events of two types, let $x_i= x$  for red bad events and $x_i=y$ for blue bad events. We shall assign the values to $p, t, x, $ and $y$ such that 
$$p^6 \leq (1-x) x^{6n^4} y^{n^{2t}} ~ \mbox{   and ~}~ (1-p)^{t^2} \leq (1-y) x^{n^2t^2} y^{n^{2t}}.$$

%
%
%
%

\noindent
Let $\epsilon, \epsilon'$ be small positive constants, say $\epsilon<1/6$, $\epsilon'<1/(6\epsilon)-1.$  Let $$t= n^{1-\epsilon}, ~ x= 1- n^{-5}, ~ y = 1 - n^{-2n^{1-\epsilon}}, ~ p = n^{-1+ \epsilon (1+ \epsilon')}.$$

\noindent
We shall be using the fact that $(1-s) \approx  e^{-s}$ for small $s$. Then, for large $n$ we have

\begin{eqnarray*}
p^6 & \approx & n^{-6+ 6 \epsilon(1 +\epsilon')},\\
(1-p)^{t^2} & \approx & e^{-n^{-1+\epsilon (1+ \epsilon')} \cdot n^{2-2\epsilon}} = e^{-n^{1-\epsilon(1-\epsilon')}},\\
(1-x) x^{6n^4} y^{n^{2t}}& \approx & n^{-5} e^{-n^{-5}6n^4} e^{-1} \geq  e^{-1} n^{-5},\\
(1-y) x^{n^4t^2} y^{n^{2t}} & \approx & n^{-2n^{1-\epsilon}}  e^{-n^{-5}n^4n^{2-2\epsilon}} e^{-1}.
\end{eqnarray*}

\noindent
Thus $p^6 \leq (1-x) x^{6n^4} y^{n^{2t}} ~ \mbox{   and ~}~ (1-p)^{t^2} \leq (1-y) x^{n^2t^2} y^{n^{2t}}.$
Therefore, by the LLL there is an edge-coloring of $K_{n,n}$ with no red $C_6$'s and no blue $K_{n^{1-\epsilon}, n^{1-\epsilon}}$.

\vskip 1 cm 
To see the result for $C_8$, we closely follow the  above argument and   choose the parameters similarly and define red bad event corresponding to a red $C_8$ and blue bad event as before.

\begin{itemize}

\item[-]
The probability of a red bad event is $p^8$. The probability of a blue bad event is $(1-p)^{t^2}$.

\item[-]
Each red bad event depends on at most $8 \binom{n}{3}^2 c  \leq Cn^6$ other red bad events (this corresponds to the number of $C_8$'s sharing an edge with a given $C_8$).

\item[-]
Each red bad event depends on at most $8 \binom{n}{t-1}^2 \leq n^{2t}$ blue bad events (this corresponds to the number of $K_{t,t}$'s sharing an edge with a given $C_8$).

\item[-]
Each blue bad event depends on $t^2\binom{n}{3}^2 c\leq  Ct^2n^6$ red bad events (this corresponds to the number of $C_8$'s sharing an edge with a given $K_{t,t}$).

\item[-]
Each blue bad event depends on $t^2\binom{n}{t-1}^2 4\leq  n^{2t}$ blue bad events (this corresponds to the number of $K_{t,t}$'s sharing an edge with a given $K_{t,t}$).
\end{itemize}

\noindent
Let $\epsilon, \epsilon'$ be a small positive constants, say $\epsilon<1/16$, $\epsilon' < 1/(16\epsilon)-1$.
Let $$t= n^{1-\epsilon}, ~ x= 1- n^{-7.5}, ~ y = 1 - n^{-2n^{1-\epsilon}}, ~ p = n^{-1+\epsilon(1+\epsilon')}.$$

\noindent
Then, for large $n$ we have
\begin{eqnarray*}
p^8 & \approx & n^{-8+ 8 \epsilon(1+\epsilon')},\\
(1-p)^{t^2} & \approx & e^{-n^{-1+\epsilon (1+ \epsilon')} \cdot n^{2-2\epsilon}} = e^{-n^{1-\epsilon(1-\epsilon')}},\\
(1-x) x^{Cn^6} y^{n^{2t}}& \approx & n^{-7.5} e^{-n^{-7.5}Cn^6} e^{-1},\\
(1-y) x^{Cn^6t^2} y^{n^{2t}} & \approx & n^{-2n^{1-\epsilon}}  e^{-Cn^{-7.5}n^6n^{2-2\epsilon}} e^{-1}.
\end{eqnarray*}

\noindent
Thus  $p^8 \leq (1-x) x^{Cn^6} y^{n^{2t}} ~ \mbox{   and ~}~ (1-p)^{t^2} \leq (1-y) x^{n^6t^2} y^{n^{2t}}.$
Therefore, by the LLL there is an edge-coloring of $K_{n,n}$ with no red $C_8$'s and no blue $K_{n^{1-\epsilon}, n^{1-\epsilon}}$. \\

Now, let $H$ be a bipartite graph containing $C_{2k}$, $k\in \{2,3,4\}$. 
Consider an edge-coloring of $K_{n,n}$ with no red $C_{2k}$ and no blue $K_{n^{1-\epsilon}, n^{1-\epsilon}}$. 
Let $G$ be a graph formed by the red edges. Then $G$ does not contain $C_{2k}$ and thus does not contain $H$, that implies that it does not contain an induced copy of $H$. In particular $G$ does not have $K_{4,4}$. 
On the other hand the bipartite complement of $G$ does not contain $K_{n^{1-\epsilon}, n^{1-\epsilon}}$. Thus, for sufficiently large $n$, 
$\tih(G) \leq n^{1-\epsilon}$.
\end{proof}

\begin{theorem}\cite{EHP}\label{thm:EHP}
Let $H$ be a bipartite graph with  parts of sizes $k$ and $l$, $2\leq k\leq l$.  Let $G$ be a bipartite graph with parts of sizes $n$, $n\geq l^{k}$.
Then either $G$ is $H$-free or $\tih(G) \geq t$, where $t = \lfloor (\frac nl)^{1/k} \rfloor $.
\end{theorem}

We need the following Lemma. 
\begin{lemma}\label{lem:smartembedding}
	Let $G$ be an $(l+1)$-partite graph with vertex classes $U_1,\ldots,U_l, V$, $|U_i| \geq t^{m}$, $|V| \geq tl$, for some integers $l, t, m \geq 2$.  Let  $\widetilde{\alpha}(U_i, V)$ and $\widetilde{\omega}(U_i, V)$  denote  $\widetilde{\alpha}$ and  $\widetilde{\omega}$ of the bipartite subgraph of $G$ induced by $(U_i, V)$. Let  $\widetilde{\alpha}(U_i, V) < t, \widetilde{\omega}(U_i, V) < t$ for all $i \in [l]$. 
	Then for any map $f:[l] \to \{0,1\}$, there exists a vertex $v \in V$, such that  
	\begin{align*}
	|N(v) \cap U_i | \geq t^{m-1},\qquad &  f(i) = 1,  \\
	|U_i\setminus N(v) | \geq t^{m-1},\qquad & f(i) = 0.
	\end{align*}
\end{lemma}
\begin{proof}
	Let $G$ be an $(l+1)$-partite graph as in the statement and fix a function $f:[l] \to \{0,1\}$.  
	We prove the statement by contradiction.\\
	Assume there is no such vertex $v \in V$. Then for every $v \in V$, there must be at least one index $i_v \in [l]$, such that $U_{i_v}$ is \emph{bad} for $v$, meaning that 
	\begin{align*}
	&|N(v) \cap U_{i_v} | \leq t^{m-1} - 1 , &  \text{if } f(i) = 1,  \\
	\text{or }\qquad&|U_{i_v}\setminus N(v) | \leq t^{m-1} -1, & \text{if } f(i) = 0.
	\end{align*}
	Since there are only $l$ sets, we have a set $U_i$, that is bad for $ \frac{|V|}{l} \geq t$ vertices in $V$. Choose $V' \subseteq V$ such that $|V'| = t$ and $i_v = j \ \forall v, w \in V'$. \\
	We now distinguish two cases: \\
	Case 1: $f(j) = 0$. Consider the subset $U' \subseteq U_j$ of vertices, that are adjacent to all vertices in $V'$, so $U' = \{u \in U_j\ |\ uv \in E(G)\ \forall v \in V' \}$. Since every vertex in $V'$ is non-adjacent to at most $t^{m-1} -1$ vertices, we obtain $|U'| \geq |U_j| - t(t^{m-1} -1) \geq t^m - (t^m - t) \geq t$. Thus, the pair $(U',V')$ contains a copy of $K_{t,t}$, which contradicts our assumption of $\widetilde{\omega}(U_j, V) < t$. 	 \\
	Case 2: $f(j) =1$. Consider the subset $U' \subseteq U_j$ of vertices, that have no neighbour in $V$, so $U'  =\{u \in U_j\ |\ uv \not\in E(H)\ \forall v \in V'\}$. Since every vertex in $V'$ is adjacent to at most $t^{m-1} -1$ vertices, we obtain $|U'| \geq |U_j| - t(t^{m-1} -1) \geq t^m - (t^m - t) \geq t $. Thus, the pair $(U',V')$ contains an empty bipartite graph of size $t$, which contradicts our assumption of $\widetilde{\alpha}(U_j, V) < t$.\\
	Hence, for every $f$ we find a vertex $v$, which is good for all sets $U_i$. 	
\end{proof}

\begin{proof}[Proof of Theorem \ref{thm:EHP}]
Let $H = (X\cup Y, E_H)$ be a bipartite graph with parts $X = \{x_1,\ldots,x_l\}$, $Y = \{y_1,\ldots,y_k\} $ with $2 \leq k \leq l$, let $n_0 = l^{k-1}$.  Assume that  $\widetilde{\alpha}(U,V) < t$ and  $\widetilde{\omega}(U,V) < t$. We show how to find an induced copy of $H$. 
Note that from the choice of $t$ and $n$, we have that $n\geq t^k l$ and $t\geq l$.\\

Partition $U$ into $l$ subsets $U_1,\ldots, U_l$, each of size at least $t^k$. Partition $V$
 into $k$ subsets $V_1,\ldots, V_k$  each of size at least $t^k$.  Since $t\geq l$, we have that $|V_i| \geq tl^{k-1}\geq tl$, for all $i \in [k]$.
We shall apply Lemma \ref{lem:smartembedding}   $k$ times to obtain subsets $U_i \supseteq U_i^1\supseteq \cdots \supseteq U_i^k$ such that we can embed $x_i \in U_i^k$ and $y_j \in V_j$ for $i=1,\ldots, l$ and $j = 1,\ldots, k$. \\

\noindent
In step 1, we apply Lemma \ref{lem:smartembedding} to the sets $U_1,\ldots,U_l,V_1$ with $m = k $ and 
	\begin{equation*}
	f_1 : [l] \to \{0,1\}, \ i \mapsto \left\{ 
					\begin{array}{ll}
					1, & x_iy_1 \in E_H \\
					0, & x_iy_1 \not\in E_H
					\end{array}
									\right.,
	\end{equation*}
	to find a vertex $v_1 \in V$ and subsets \begin{equation*}
U_i^1 = \left\{
\begin{array}{ll}
N(v_1) \cap U_i, & x_iy_1 \in E_H \\
U_i\setminus N(v_1), & x_iy_1 \not\in E_H
\end{array}							
\right. , 
	\end{equation*} such that $|U_i^1| \geq t^{k-1}  $ for $i \in [l]$. \\

\noindent	
	Assume that after step $j$ we have subsets $U_1^j,\ldots, U_l^j$, $|U_i^j| \geq t^{k-j} $. \\
	If $j < k$, we apply the Lemma again, to the sets $U_1^j,\ldots, U_l^j, V_{j+1}$ with $ m = k-{j-1} \geq 2 $ and 
	\begin{equation*}
	f_j : [l] \to \{0,1\},\  
	i \mapsto \left\{ 
	\begin{array}{ll}
	1, & x_iy_{j+1} \in E_H \\
	0, & x_iy_{j+1} \not\in E_H
	\end{array}
	\right.,
	\end{equation*} 
	to find a vertex $v_{j+1} \in V_{j+1}$ and subsets 
	\begin{equation*}
		U_i^{j+1} = \left\{
	\begin{array}{ll}
	N(v_1) \cap U_i^j, & x_iy_{j+1} \in E_H \\
	U_i^j\setminus N(v_1), & x_iy_{j+1} \not\in E_H
	\end{array}							
	\right.,
	\end{equation*}
	 such that $|U_i^{j+1}| \geq t^{(k-j)-1}= t^{k-(j+1)}  $ for $i \in [l]$. \\
	 
	 \noindent
	We finish after $k$ steps, and by our choice of $n, t$, we still obtain $|U_i^k| \geq 1$, $i \in [l]$. Thus, we have found vertices $v_1,\ldots,v_k$ where we can embed $\{x_1,\ldots,x_k\}$ and nonempty sets of candidates $U_i^k$, in which we can embed $Y$.  This concludes the proof.
\end{proof}

\vskip 1cm
\noindent
{\bf Acknowledgments} The authors thank Jacob Fox for bringing the connections between the considered problem and the VC-dimensions of sets systems to their attention. The first author thanks the Department of Mathematics at Stanford University for the hospitality.


\begin{thebibliography}{99}

\bibitem{CR} Yair Caro and Cecil Rousseau. {\it Asymptotic bounds for bipartite Ramsey numbers}. The Electronic Journal of Combinatorics, 8(1):17, 2001.

\bibitem{C} Maria  Chudnovsky. {\it The Erd\H{o}s-Hajnal conjecture -- a survey}.  Journal of Graph Theory 75 (2014)
178--190.

\bibitem{EH} Paul Erd\H{o}s and Andr\'as Hajnal. {\it Ramsey-type theorems}. Discrete Applied Mathematics 25(1989) 37--52.

\bibitem{EHP} Paul Erd\H{o}s, Andr\'as Hajnal, and J\'anos Pach. {\it A Ramsey-type theorem for bipartite graphs}.Geombinatorics, 10(DCG-ARTICLE-2000-001) ~(2000) 64--68. 


\bibitem{FPS}  Jacob Fox, J\'anos Pach,  and Andrew Suk. {\it Erd\H{o}s-Hajnal conjecture for graphs with bounded VC-dimension. }33rd International Symposium on Computational Geometry, Art. No. 43, 15 pp., LIPIcs. Leibniz Int. Proc. Inform., 77, Schloss Dagstuhl. Leibniz-Zent. Inform., Wadern, 2017.


\bibitem{KPT} Daniel Kor\'andi, J\'anos Pach, and Istv\'an Tomon. {\it Large homogeneous submatrices}. arXiv:1903.06608


\end{thebibliography}
\end{document}